\documentclass[a4paper,11pt]{amsart}
\usepackage{mathptmx}
\usepackage{geometry}
\usepackage[utf8]{inputenc}
\usepackage{setspace}
\usepackage{amsmath,amsthm,amsfonts}
\usepackage{amsbsy}
\usepackage[T1]{fontenc}
\usepackage{graphicx}
\usepackage{amssymb}
\usepackage[all,cmtip]{xy}
\usepackage{tikz}
\usepackage{indentfirst}
\usepackage[mathscr]{euscript}
\usepackage{mathrsfs}
\DeclareMathAlphabet{\pazocal}{OMS}{zplm}{m}{n}

\usetikzlibrary{matrix,arrows,decorations.pathmorphing}

\newtheorem{thm}{Theorem}[section]
\newtheorem{cor}[thm]{Corollary}
\newtheorem{lem}[thm]{Lemma}
\newtheorem{prop}[thm]{Proposition}
\theoremstyle{definition}
\newtheorem{defn}[thm]{Definition}
\theoremstyle{remark}

\newtheorem{example}{Example}[section]

\newcommand{\Z}{\mathbb Z}
\newcommand{\C}{\mathbb C}
\newcommand{\A}{\EuScript A}

\title[Length and decomposition]{Length and decomposition of the cohomology of the complement to a hyperplane arrangement}
\author[R.~B\o gvad]{Rikard B\o gvad}
\address{Department of Mathematics, Stockholm University, SE-106 91
Stockholm,         Sweden}
\email {rikard@math.su.se }

\author[I.~Gon\c calves]{Iara Gon\c calves}
\address{Department of Mathematics, Stockholm University, SE-106 91
Stockholm,         Sweden}
\email{iara@math.su.se}

\begin{document}

\begin{abstract}
Let $\A$ be a hyperplane arrangement in $\C^n$.
We prove in an elementary way that the number of decomposition factors as a perverse sheaf of the direct image $Rj_*\C_{\tilde U}[n]$ of the constant sheaf on the complement ${\tilde U}$ to the arrangement is given by the Poincar\'e polynomial of the 
arrangement. Furthermore we describe the decomposition factors of $Rj_*\C_{\tilde U}[n]$ as certain local cohomology 
sheaves and give their multiplicity. These results are implicitly contained, with different proofs, in Loiijenga 
\cite{EL}, Budur and Saito\cite{BuS}, Petersen\cite{Pet} and Oaku \cite{Oa}.
\end{abstract}

\maketitle

\section{Introduction}

Let $Perv(X)$ be the category of perverse sheaves on a complex smooth algebraic variety $X$ (with respect to the middle perversity). Recall that $Perv(X)$ is an abelian category, where every object has a finite decomposition series. 
If $\pi:X\to Y$ is a proper map and $M\in Perv(X)$ a perverse sheaf, then by the decomposition theorem (see \cite{CM}) $N:=R\pi_*M$ is semi-simple. This is not true if $\pi$ is not proper, and it is then natural to ask for properties 
of the decomposition series of $R\pi_*M$.  

In this note we give in an elementary way such a description in a special situation: $M=\C_{\tilde U}$ is the constant sheaf and $j:{\tilde U} \to \C^n $ is the inclusion of the complement ${\tilde U}$ to a hyperplane arrangement 
$\A$ in $\C^n$. The cohomology of this sheaf has been subject to much study, see the book by Orlik and Terao \cite{OT}. Our main result is that the number of perverse decomposition factors of the direct image $Rj_*\C_{\tilde U}[n]$ 
equals the Poincar\'e polynomial of the arrangement, using a Mayer-Vietoris sequence. In addition our proof gives in an explicit form the decomposition factors, and using a result of Jewell~\cite{J}, their multiplicity. 

By the Riemann-Hilbert correspondence these results relate to considering $\mathcal O_{\tilde U}$ as a module over the Weyl algebra $A_n$.  Results similar via this correspondence to our results are contained in a recent work 
of Oaku \cite{Oa}.

After the completion of this work, we were made aware that 
several authors have published work which implicitly contains our results.
Loiijenga 
\cite[2.4.1]{EL}, constructs(for an affine hyperplane arrangement) a complex $K$ quasi-isomorphic to $Rj_!\C_{\tilde U}[n]$, and determines the associated graded object of the weight filtration on $K$, as a direct sum with given multiplicity of our $N_F$. Since the weight filtration respects perversity(see \cite{BBD}) and the $N_F$ are irreducible perverse sheaves, our results follow. One may also extract them from similar constructions and determination of the weight filtration in Budur-Saito \cite[1.7-9]{BuS}(on projective hyperplane arrangements). Finally, Petersen \cite[Thm. 1.1, Example 3.10]{Pet} describes in general a spectral sequence associated to a stratified space that is compatible with the perverse filtration. In the case of hyperplane arrangements this sequence degenerates and gives Looijenga's result. 

We hope that an explicit statement of the length and decomposition of $Rj_*\C_{\tilde U}[n]$, and an elementary proof that only uses deletion-and-restriction and basic properties of perverse sheaves, still merits interest.

One may ask for similar results for an arbitrary locally constant sheaf. The case of central line arrangements and a rank 1 locally constant sheaf is treated in \cite{AB1,BG}. See also \cite{Bu2} for more general results.

\vspace{0.2cm}

\section{Notation and preliminaries}
Let $\A$ be an affine hyperplane arrangement of $m+1$ hyperplanes $H_0, \ldots, H_m$ in $\C^n$. The arrangement defines a stratification $\Sigma=\Sigma_\A$ of 
$\C^n$ by {\it flats}, that is intersections of subsets of hyperplanes. As a general reference on hyperplane arrangements we use \cite{OT}.

For the complex variety $X=\C^n$ of dimension $n$ and the stratification $\Sigma$ by flats, let $D_\Sigma(X)$ denote the derived category of complexes of sheaves that are constructible with respect to $\Sigma$. 
We let $Perv(X)$ be complexes of sheaves in $D_\Sigma(X)$ that are perverse with respect to this stratification, for the middle perversity $p$. As a general reference on intersection cohomology we use \cite{BBD}. Recall that if $F$ is a flat of $\C^n$, 
then $p(F)= -dim _{\C}(F)$. For this perversity, a locally constant sheaf placed in degree $-n$ is a perverse sheaf (\cite{BBD}, pp.63-64). In particular the constant sheaf $\C[n]$ is perverse. We will denote by $\C_F$ the constant sheaf of rank one on the flat $F$.
All subvarieties of $\C^n$ will be assumed to have a stratification induced by the stratification of $\C^n$, or possibly by a subarrangement. Consequently, constructibility with respect to the filtration will be respected by functors such as direct images, and in the notation we will suppress reference to the stratification. 

The following standard lemma from \cite{BBD} is useful to describe exact sequences in $Perv(X)$.

\begin{lem} 
\label{lem:critexactseq}Suppose that 
$$N \rightarrow M \rightarrow L \rightarrow$$
is a distinguished triangle in $D(X)$. 
Then 
$$ 0\rightarrow N \rightarrow M \rightarrow L \rightarrow 0,$$ is an exact sequence in $Perv(X)$ in either of the following two situations:
\begin{itemize}
\item[(i)]If $N,M,$ and $L$ are perverse sheaves.
\item[(ii)] If $N,M$ are perverse sheaves and $N\to M$ is an injection in $Perv(X)$.
 \end{itemize}
\end{lem}

\begin{proof}
This follows from the description of kernels and cokernels in \\
$Perv(X)$, in terms of the t-structure of $Perv(X)=D^{\geq 0}\cap D^{\leq 0}$ (see Prop. 1.2.2 and Thm. 1.3.6 in \cite{BBD}). (i) The kernel of the map $\alpha: N\rightarrow M$ is given by 
$(\tau_{\leq -1}L)[-1]$, which is $0$ since $L\in D^{\geq 0}$. The cokernel of $\alpha $ is $\tau_{\geq 0}L=L$, and this identity also gives (ii). This proves the lemma. 
\end{proof}

\begin{defn}[\cite{KS}, Ex.8.20]
Let $\mathcal C$ be an abelian category. An object $M$ in $\mathcal C$ is irreducible (or simple) if it is not isomorphic to $0$ and any subobject of $M$ is either $M$ or $0$.
A sequence
$$M=M_0 \supset M_1 \supset \ldots \supset M_{n-1} \supset M_n = 0$$
is a decomposition series if the quotient $M_i / M_{i+1}$ is irreducible for all $i$ with $0 \leq i <n$.
\end{defn}

Every perverse sheaf has a finite decomposition series whose successive quotients are irreducible perverse sheaves (\cite{BBD}).

We call the integer $n$ the length of the object $X$. We will denote the number of factors in a decomposition series of an object $M$ by $c(M)$ and so $c(M)=n$. 
In the situation of Lemma \ref{lem:critexactseq}
$$c(M)=c(N) + c(L).$$

\vspace{0.2cm}

\section{Mayer-Vietoris}
\label{sec:MV}

We will use the idea of deletion and restriction, to set up an inductive description of $Rj_{*}\C_{\tilde U}[n]$ using Mayer-Vietoris sequences. 
Define the subarrangements $\A':= \A - H_0$ in $\C^n$ and $$\A'':= \{ H_0 \cap H \mid H \in \A'\ \mbox{and} \ H_0 \cap H \neq H_0 \},$$
the restriction of $\A$ to $H_0$.
Consider the following sets
$$U=  \C^n\setminus H_0 \quad , \quad V= \C^n \setminus \cup_{i=1}^m H_i$$

$$\tilde U:=U \cap V = \C^n \setminus \cup_{i=0}^m H_i \quad , \quad U \cup V = \C^n \setminus \left( H_0 \cap (\cup_{i=1}^m H_i) \right)$$
and inclusions:
$$j_Y: Y \hookrightarrow \C^n \quad , \quad i_Y: \C^n \setminus Y \hookrightarrow \C^n$$
where $Y$ denotes one of the four sets $U, V, U \cap V,$ and $ U \cup V$. Note that they all have dimension $n$, and hence $\C[n]$ is a perverse sheaf in the corresponding derived category.

Consider in $D(\C^n)$ the distinguished triangle corresponding to the Mayer-Vietoris sequence associated to $$U \cap V = \C^n\setminus \cup_{i=0}^m H_i $$ (see \cite{KaS}, pp.94,114).

\begin{equation}
 \label{MVR}
  Rj_{U \cup V*}j_{U \cup V}^* \C[n] \rightarrow Rj_{U*}j_U^*\C[n] \oplus Rj_{V*}j_V^* \C[n] \rightarrow Rj_{U \cap V*}j_{U \cap V}^* \C[n] \rightarrow 
\end{equation}

When the inclusion $j$ is an affine and quasi-finite morphism, $Rj_*$ is a t-exact functor. In the sequence above, while $j_{U}$, $j_{V}$, $j_{U \cap V}$ are affine and quasi-finite, 
$j_{U \cup V}$ is not, therefore (\ref{MVR}) is not an exact sequence of perverse sheaves.  However, we have some canonical irreducible subobjects.

\vspace{0.2cm}

\begin{lem}
\label{injg}
The adjunction morphism $g_Y: \C[n] \rightarrow Rj_{Y*}j_Y^*\C[n]$ is an injective morphism in $Perv(\C^n)$, where $Y$ is any of the three affine subsets $U, V$ and $U \cap V$ of $\C^n$.
\end{lem}

\begin{proof}
Since $\C[n]$ is an irreducible perverse sheaf, either $ker(g_Y)=0$ or $ker(g_Y)=\C[n]$, implying that the morphism $g_Y$ is either injective or zero, respectively.
Suppose $ker(g_Y)=\C[n]$ so that $g_Y=0$. Then
$$j^*_Y(g_Y) :j_Y^*\C[n] \rightarrow j_Y^*Rj_{Y*}j_Y^*\C[n],$$
is also zero.
But $j_Y^*Rj_{Y*}j_Y^*\C[n] = j_Y^*\C[n]$, and $j^*_Y(g_Y)$ is the identity morphism. 
We conclude that our initial assumption was not correct, therefore $ker(g_Y)=0$ and $g_Y$ is injective.
\end{proof}

If $j_Y: \ Y \to \C^n$ is one of our open inclusions, we let $i_Y:\ \C^n\setminus Y \to \C^n$ be the corresponding closed inclusion. Recall (\cite{BBD}, pp.51) that there is a distinguished triangle
\begin{equation}
 \label{iFj}
   {}i_*{}i^!\mathcal F \rightarrow \mathcal F \rightarrow {}j_*{}j^* \mathcal F\rightarrow
\end{equation}
(irrespective of whether these complexes are perverse).

Lemma \ref{lem:critexactseq} now implies that for $Y$ as in the preceding lemma 
\begin{equation}
\label{eq:loccoh}
{}i_{Y*}i^!_Y \mathcal \C[n+1] = Rj_{Y*}j^*_Y \C[n] / \C[n],
\end{equation}
is in $Perv(\C^n)$. 

\vspace{0.2cm}

Let $\tilde{i}_Y := i_{Y^*}i^!_Y$ and $\tilde{j}_Y:=Rj_{Y*}j^*_Y$. From \eqref{MVR} and \eqref{iFj} we have in $D(\C^n)$ the following diagram:

\begin{equation}
\xymatrix{
\tilde{i}_{U}\C[n+1] \oplus \tilde{i}_{V}\C[n+1] \ar[r] & \tilde{i}_{U \cap V }\C[n+1] \ar[r] & \tilde{i}_{U \cup V }\C[n+2] \\
\tilde j_{U}\C[n] \oplus \tilde j_{V}\C[n] \ar[r] \ar[u] & \tilde j_{U \cap V}\C[n] \ar[r] \ar[u] & \tilde j_{U \cup V}\C[n+1] \ar[u] \\
\C[n] \oplus \C[n] \ar[u] \ar[r] &  \C[n] \ar[u] \ar[r] & \C[n+1] \ar[u]
}
\end{equation}

All vertical and horizontal sequences are exact triangles.
The lower vertical maps are the adjunction morphisms, and the remaining vertical sequences are applications of \eqref{iFj}.

We will now show that the first line of this diagram is an exact sequence in $Perv(\C^n)$.
\begin{prop}
\label{prop:seq2}
\begin{equation}
\label{eq:prop}
i_{U*}i^!_U\C[n+1] \oplus i_{V*}i^!_V\C[n+1] \rightarrow  i_{U \cap V *}i^!_{U \cap V}\C[n+1] \rightarrow i_{U \cup V *}i^!_{ U \cup V}\C[n+2]
\end{equation}
is a short exact sequence of perverse sheaves. 
\end{prop}

\begin{proof} By Lemma \ref{lem:critexactseq} it suffices to prove that all terms are perverse sheaves.
 The objects in the first two terms are perverse sheaves, since they are quotients of perverse sheaves, by Lemmas \ref{lem:critexactseq} and  \ref{injg}.

The morphism $i_{U \cup V }$ can be decomposed as:
$$ i_{U \cup V }:H_0\cap (\cup_{i=1}^m H_i )\xrightarrow{i_1} H_0 \xrightarrow{i_2} \C^n$$
and consequently $i_{U \cup V *}i^!_{ U \cup V}\C[n+2]=i_{2*}i_{1*}i_1^!i_2^!\C[n+2].$

The real codimension of the linear subspace $H_0$ in $\C^n$ is $2$, and hence $i_2^!\C[n+2]=\C_{H_0}[n]$.

Let $j_{1}: H_0 - \left( H_0 \cap (\cup_{i=1}^m H_i) \right) \rightarrow H_0$ be the open inclusion corresponding to $i_{1}$. 
We have a distinguished triangle
$$\C_{H_0}[n-1] \rightarrow Rj_{1*}j^*_1\C_{H_0}[n-1] \rightarrow i_{1*}i_1^!\C_{H_0}[n]\rightarrow$$
Since $j_{1}$ is affine and quasi-finite and so ${j}^*_1$ and $j_{1*}$ are t-exact and 
$\C_{H_0}[n-1]$ is perverse, the two first sheaves in this sequence are perverse. By Lemma \ref{injg} the first map is an injection and hence 
by Lemma \ref{lem:critexactseq} $i_{1*}i_1^!\C_{H_0}[n]$ is a perverse sheaf in $D(H_0)$, as well as the cokernel of the injection.

Recalling that the direct image $i_{2*}$ is t-exact, we reach the conclusion that $i_{2*}(i_{1*}i_1^! \C_{H_0}[n])=i_{2*}i_{1*}i_1^!i_2^!\C[n+2]$ is a perverse sheaf in $\C^n$. This finishes the proof of the theorem.
\end{proof}

\begin{cor} 
\label{cor:c}
The length $c$ of the decomposition series of the sheaves in the proposition satisfies
$$c(i_{U \cap V *}i^!_{U \cap V}\C[n+1])=1+ c(i_{V*}i^!_V\C[n+1])+c(i_{U \cup V *}i^!_{ U \cup V}\C[n+2]).$$
\end{cor}

\begin{proof}The complex $i_{U*}i^!_U\C[n+1]=i_{2*}\C_{H_0}[n-1]$ is irreducible (this follows easily from the fact that $i_{2*}$ is $t$-exact).
Hence this corollary is a direct consequence of the exact sequence in the proposition.
\end{proof}

\vspace{0.2cm}

\section{Length}  

Let $\A$ be an arrangement in $\C^n$, and $j:\tilde U\to \C^n$ the inclusion of the open complement of the hyperplanes. We will now see that \eqref{eq:prop} means that the number of decomposition factors behaves as the Poincar\'e polynomial of the set of flats. We follow \cite{OT}.
Let $L=L(\A)$ be the set of nonempty intersections of elements of $\A$, i.e. flats. Define a partial order on $L$ by
$$F \leq G \Longleftrightarrow G \subseteq F.$$

\begin{defn}
Let $F, G, K \in L$. The M\"obius function $\mu_{\A}: L \times L \rightarrow \Z$ is defined recursively as:
\[
\mu(F,G)= 
\left\{\begin{array}{l}
       1 \ \mbox{if} \ F=G        \\
       - \sum_{F \leq K \leq G}\mu(F,K) \ \mbox{if} \ F < G    \\
       0 \ \mbox{otherwise}
       \end{array}
       \right. \]
\end{defn}

For $F \in L$, we define $\mu(F)=\mu(\C^n,F)$.

\begin{defn}[\cite{OT}, Def.2.48]
Let $\A$ be an arrangement in $ \C^n$ with intersection poset $L$ and M\"obius function $\mu$. Let $t$ be an indeterminate. Define 
the Poincar\' e polynomial of $\A$ by
$$\Pi(\A,t)= \sum_{F \in L} \mu(F)(-t)^{codim \ F}$$
\end{defn}

\begin{thm} Let $\A$ be a hyperplane arrangement with hyperplanes $H_i, \ i=0,...,m$. Let $j:\tilde U:=\C^n\setminus \cup_{i=0}^m H_i\to \C^n$ be the inclusion of the complement to the arrangement, and $\C_{\tilde U}$ the constant sheaf on $\tilde U$. Then
\label{thm:main}
$$c(Rj_*\C_{\tilde U}[n])= \Pi(\A,1) = \sum_{F \in L(\A)} \lvert \mu(F) \rvert.$$ 
\end{thm}

\begin{proof} We use the notation of the previous section. Note that $\tilde U=U\cap V$. By Lemma \ref{injg} and \eqref{eq:loccoh} the theorem is equivalent to showing that 
$$c(i_{U \cap V *}i^!_{U \cap V}\C[n+1])= \Pi(\A,1)-1.$$
We make induction on the number of hyperplanes in $\A$.
If $H_0$ is the only element of $\A$, then $i_{U \cap V*}i^!_{U \cap V} \C[n+1]= i_{U*}i^!_U \C[n+1]$ is an irreducible object. Therefore the equality holds if $\A= \{ H_0\}$, given that clearly $\Pi(\A,1)=2$.
Let $\A'$ and $\A''$ be as described before (section \ref{sec:MV}). By the induction hypothesis, we have
$$c(i_{U*}i^!_U\C[n+1] \oplus i_{V*}i^!_V\C[n+1])= 1+ \Pi(\A',1)-1=\Pi(\A',1)$$
$$c(i_{U \cup V *}i^!_{ U \cup V}\C[n+1])= \Pi(\A'',1)-1$$
Hence according to Corollary \ref{cor:c}
$$c(i_{U \cap V *}i^!_{U \cap V}\C[n+1]) = \Pi(\A',1)+ \Pi(\A'',1)-1$$
Since $(\A, \A', \A'')$ form a triple of arrangements, we can use Theorem 2.56 of \cite{OT},  that states
$$\Pi(\A,t)= \Pi(\A',t)+ t\Pi(\A'',t)$$
concluding that
$$c(i_{U \cap V *}i^!_{U \cap V}\C[n+1])= \Pi(\A,1)-1.$$
By Lemma \ref{injg}, this proves the first equality of the theorem. The second, which says that the Poincar\'e polynomial has positive coefficients, follows from \cite[Thm 2.47] {OT}.
\end{proof}

\begin{example}
Consider the braid arrangement $\mathcal B_n$ in $\C^n$ consisting in $\dbinom{n}{2}$  hyperplanes $B_{ij}$ such that
$$B_{ij}= \{x \in \C^n \mid x_i = x_j, \ 1 \leq i < j \leq n \}$$
The Poincar\'e polynomial for $\mathcal B_n$ is given by (see \cite{OT}, Prop. 2.54) 
$$\Pi(\mathcal B_n, t)=(1+t)(1+2t) \ldots (1+(n-1)t).$$
So the length of $Rj_{*}\C_{\tilde U}[n]$ is
$$c(Rj_*\C_{\tilde U}[n]) = \Pi(\mathcal B_n, 1) =(1+1)(1+2) \ldots (n)= n!$$
\end{example}

\section{Decomposition factors}

We can refine the computation of the length, so as to describe the decomposition factors. They will be of the following form.
Let $F$ be a flat and $i:F\to \C^n=X$ the inclusion. We associate to $F$ the irreducible perverse sheaf
$$N_F=i_*\C_F[\dim_\C F]\in Perv(X).$$ 
Since $i_*$ is t-exact it is clear that $N_F$ is perverse, and the exactness also implies the irreducibility of $N_F$.
For example $N_X=\C_X[n]$. 
Since $F$ is a subvector space, the local cohomology $i_*i^!\C_X$ has just has one non-zero cohomology group, equal to $\C$ and placed  in the real codimension $2n-2\dim_\C F$ of $F$(see \cite{Iv}).
Thus $i_*i^!\C_X=i_*\C_F[-(2n-2\dim_\C F)]$ and so
$N_F= i_*i^!\C_X[2n-\dim_\C F]\in D(X)$. For example, if $F$ is a hyperplane $N_F=i_*\C_F[n-1]= i_*i^!\C_X[n+1]$.

We will describe the decomposition series of $Rj_*\C_{\tilde U}[n]$, as an element in the Grothendieck group $G(\A)$ of $Perv(X)$. Recall that this is the free abelian group on symbols $[K]$, one for each perverse sheaf 
$K$ in $Perv(X)$, modulo the relations $[M]+[L]=[N]$ for each short exact sequence $N \to M \to L$. Clearly $G(\A)$ is a free abelian group with a basis corresponding to the set of irreducible perverse sheaves.
\begin{prop}
Let $j: \tilde U \rightarrow \C^n = X$ be the inclusion of the complement of the hyperplane arrangement.
 \begin{equation}
 \label{eq:groth}
[Rj_*\C_{\tilde U}[n]]=\sum_{F\in {L(\mathcal A)}}\vert \mu(F)\vert [N_F],
\end{equation}

where $\mu(F)$ is the M\"obius function on the intersection lattice $L(\mathcal A)$.
\end{prop} 

\begin{proof}We first note that all composition factors are of the form $[N_F]$, where $F$ is a flat. This follows by induction on the number of hyperplanes of the arrangement from the sequence \eqref{eq:prop}, as in the proof of 
Theorem \ref{thm:main}. The base cases are $\C_X[n]=N_X$ and $N_F=i_*\C_F[n-1]= i_*i^!\C_X[n+1]$, if $F$ is a hyperplane.

The formula may also be proved by induction on the number of hyperplanes in a similar way. 
Translate the right hand side of \eqref{eq:groth} and define
$$Q(\A,1)=\sum_{F\in{L(\mathcal A)}}\vert \mu(F)\vert [N_F] - [N_{\C^n}]$$
($Q(\A,1)$ is an evaluation of an evident Poincar\'e polynomial in $G(\A)$.)

Then use the notation of Proposition \ref{prop:seq2}. Set 
$M_Y:=i_{Y*}i^!_Y\C[n+1] $ if $Y=U,V,U \cap V,U\cup V$, and note that the proposition (for $U\cap V\to \C^n$) is equivalent to proving 
$[M_{U\cap V}]=Q(\A,1)$. Note that $[M_U]=[N_{H_0}]$.

Clearly \eqref{eq:prop} implies that
$$
[M_U]+ [M_V] +[M_{U\cup V}[1]]=[M_{U\cap V}].
$$
By induction we assume that the formula is true for $\A'$ and $\A''$ and hence the result will follow if we know that
$$
[N_{H_0}]+Q(\A',1)+Q(\A'',1)=Q(\A,1).
$$
This amounts to the following property of the M\"obius function of a hyperplane arrangement:
\begin{equation}
\label{eq:mobius}
\vert \mu_{\A}(F)\vert=\vert \mu_{\A'}(F)\vert  +\vert \mu_{\A''}(F)\vert
\end{equation}
using the notation of section \ref{sec:MV}. 
In Theorem 2.5 of \cite{J}, which makes a careful study of the effect of insertion-restriction on the cohomology of the complement, there is a long exact sequence that splits into short exact sequences in the case 
of hyperplane arrangements (see Prop.3.4-6(ibid.)). The ranks of the modules in that sequence are the M\"obius functions in \eqref{eq:mobius} (by Thm 3.3(ibid.)), hence we get in particular the desired result.  This 
finishes the proof.
\end{proof}
\vspace{1cm}

\subsection{Decomposition Factors of $Rj_!\C_{\tilde U}[n]$}

By Verdier duality there is also a result for the dual to $Rj_*\C_{\tilde U}[n]$, which we now describe. It uses the sheaves $D_X(N_F)= N_F$. In fact we have,
\begin{equation}
\begin{split}
D(N_F) & =D_X(i_* i^*\C_X[dim_{\C}F])= i_*D_F(i^*\C_X[dim_{\C} F])=\\ 
 & = i_*i^!(D_X\C_X[dim_{\C} F])=i_*i^!(D_X\C_X)[-dim_{\C} F]=\\
 & = i_*i^!\C_X[2n-dim_{\C} F]=i_*\C_F [dim_{\C} F] = N_F.
\end{split}
\end{equation}

(Recall that $i_*=i_!$.)
Since, as mentioned before, $i^!\C_X = \C_F[-(2n -2dim_{\C}F)]$.

\begin{cor}
$$[Rj_!\C_{\tilde U}[n]]= \sum_{F \in L(\A)}\lvert \mu(F)\rvert[N_F].$$  
\end{cor}

In the following lemma $M$ is taken as an element in the Grothendieck group $G(\A)$ of perverse sheaves in $\C^n$.

\begin{lem}
Let $M$ be a perverse sheaf in $\C^n$ and $K_i$, $i \in I$, denote the factors of its decomposition series, with symbols $[K_i]$. If
$[M]= \sum_i[K_i]$, then $[D_XM]= \sum_i [D_XK_i]$.
\end{lem}

\begin{proof}
Suppose that $M$ has only two decomposition factors, such that $0 \rightarrow K_1 \rightarrow M \rightarrow K_2 \rightarrow 0$ is an exact sequence. Then, the dualized sequence,
$$0 \rightarrow D_XK_2 \rightarrow D_XM \rightarrow D_XK_1 \rightarrow 0$$
is also an exact sequence.
Obviously, $[M]=[K_1]+[K_2]$, implies $[D_XM]=[D_XK_1]+[D_XK_2].$
By induction in the number of decomposition factors we obtain the desired result.
\end{proof}

\begin{proof}(of the Corollary)
Let $F$ be a flat, $U= X \setminus F = \C^n \setminus F$ and consider $i: F \rightarrow \C^n$.
First note that 
\begin{equation}
\begin{split}
D_X(\C_X[dim_{\C} X]) & = (D_X\C_X)[-dim_{\C}X]= \C_X[2dim_{\C} X -dim_{\C}X] = \\
& =\C_X[dim_{\C} X].
\end{split}
\end{equation}

From Verdier duality, we have 
\begin{equation}
\label{Rdual}
D_U(Rj_!\C_{\tilde U}\C[n])= Rj_*(D_{X}\C_{\tilde U}[n]) = Rj_*(D_{X}\C_{\tilde U})[-n]=Rj_*\C_{\tilde U}[n].
\end{equation}

Hence, by \eqref{Rdual} and the relation in Proposition 5.1, we have that
$$[Rj_!\C_{\tilde U}[n]]= \sum_{F \in L(\A)}\lvert \mu(F)\rvert[N_F].$$

\end{proof}

\centerline{\bf Acknowledgements}

\bigskip

\noindent
We want to thank Rolf K{\"a}llstr{\"o}m for discussions on these topics, as well as Nero Budur and Dan Petersen for directing us to 
\cite{BuS} respectively \cite{EL} and \cite{Pet}. The second author gratefully acknowledges financing
by SIDA/ISP.

\end{document}